\documentclass[10pt]{article}
\usepackage{amsfonts,amsmath,amssymb,amstext,amsbsy,amsopn,amsthm}
\usepackage{mathptmx}
\usepackage{url}
\usepackage{graphicx}
\usepackage{amsbsy}
\def\R{{\mathbb R}}
\def\Z{{\mathbb Z}}

\def\N{{\mathbf N}}
\newcommand{\itg}{\int \limits}

\newcommand{\ee}{{\rm e}}

\DeclareMathOperator{\erf}{erf}

\def\re{{\rm Re\,}}

\def\ba{\boldsymbol{\alpha}}

\def\a{{\alpha}}

\def\e{\varepsilon}
\def\eps{\varepsilon}

\def\t{\tau}
\def\h{\eta}

\def\G{\Gamma}

\newcommand{\cB}{\mathcal B}
\newcommand{\cD}{\mathcal D}

\newcommand{\cL}{\mathcal L}

\newcommand{\cO}{\mathcal O}

\newcommand{\cS}{\mathcal S}

\newcommand{\cF}{\mathcal F}

\def\bx{{\mathbf x}}
\def\by{{\mathbf y}}
\def\bm{{\mathbf m}}
\def\bk{{\mathbf k}}

\newtheorem{thm}{Theorem}[section]

\title{Fast cubature of high dimensional biharmonic potential based on Approximate Approximations}

\author{ Flavia Lanzara 
\thanks{Department of Mathematics, 
Sapienza University of Rome, 
Piazzale Aldo Moro 2, 00185 Rome, Italy. 
{\it email:} lanzara@mat.uniroma1.it}
\and Vladimir Maz'ya \thanks{Department of Mathematics, University of
Link\"oping,  581 83 Link\"oping, Sweden. }
\thanks{
Department of Mathematical Sciences, M\&O Building, University of
Liverpool, Liverpool L69 3BX, UK. {\it email:} vlmaz@mai.liu.se}
\and Gunther Schmidt
\thanks{Lichtenberger Str. 12, 10178 Berlin, Germany. {\it email:} schmidt.gunther@online.de}
}

\begin{document}
%\date{}
\maketitle
{\bf  Abstract.}  { We derive new formulas for the high dimensional biharmonic potential acting on Gaussians or Gaussians times special polynomials. These formulas can be used to construct accurate cubature formulas of an arbitrary high order which are fast and effective also in very high dimensions. Numerical tests show that the formulas are accurate and provide the predicted approximation rate \(O(h^8)\) up to the dimension \(10^7\).}

{\bf AMS subject classifications.} 65D32, 41A30, 41A63.

{\bf Keywords:} Biharmonic potential, higher dimensions, separated representations, multi\-di\-mensio\-nal convolution.
\section{Introduction}

 The present paper is devoted to the approximation of the high dimensional  biharmonic potential  
\begin{equation}\label{biharm}
\cB_n f(\bx)=\frac{\G(n/2)}{4\pi^{n/2} (n-2)(n-4)} \itg_{\R^n}\frac{f(\by)}{|\bx-\by|^{n-4}}d\by,\qquad \bx\in\R^n,\quad {n\geq 3, \, n\neq 4 ,}
\end{equation}
(cf. \cite[p.235]{mitrea} or \cite[p.100]{NW}) for integrable \(f\), by using approximate approximations (cf.\cite{MSbook} and the references therein). 
Approximate approximations allow to construct efficient high order cubature formulas for convolution integral operators even with singular kernel functions (cf. \cite{MS95}). Due to the operation number proportional to \(h^{-n}\), where \(h\) denotes size of a uniform grid on the support of the density, these methods are practical only for small \(n\). 
By combining approximate approximations with separated representations  (also called tensor structured approximations) introduced in \cite{BM,BM2}, we derive a method for approximating volume potentials which is accurate and fast  in high dimensions. \\
%%%%%%%%%%%%%%%%%%%%%%%
In the last years moderns method based on tensor structured approximations have been applied successfully to some class of multidimensional integral operators (e.g. \cite{BCFH,BFHKM,BKM,GHK,Hack, HK,HK2,Khor,Khor2}).
In this paper, for the first time to our knowledge, the cubature of the high dimensional biharmonic potential of an arbitrary high order is considered. We derive new formulas for  the biharmonic  potential acting on Gaussians or special polynomial times Gaussians. These formulas can be used to construct fast and high-order accurate cubature formulas in high dimensions.  We report on numerical  experiments which show approximation order \(O(h^8)\)  up to dimension \(n=10^7\).\\
%%%%%%%%%%
First results on the fast cubature of high dimensional harmonic potential in the framework of approximate approximations have been obtained in \cite{LMS2,LMS3}. The procedure has been applied in \cite{LMS4,LMS5} to advection-diffusion potentials and in \cite{LMS6} to parabolic problems of second order. In \cite{LMS17} our approach has been  extended to the computation of the Schr\"odinger potential, where standard cubature methods are very expansive due to the  fast oscillations of the kernel. Here we consider the problem of constructing fast cubature formulas for higher order problems.

We construct an approximation of \(\cB_n f\) if we replace \(f\) by functions with 
analytically known biharmonic potential. Specifically, we approximate the density \(f\in C^N_0(\R^n)\)
with the approximate quasi-interpolant
\begin{equation}\label{fhD}
f_{h,\cD}(\bx)=\cD^{-n/2} \sum_{\bm\in\Z^n} f(h \bm) \eta\left(\frac{\bx -h\bm}{h\sqrt{\cD}}\right) , 
\end{equation}
where \(h\) and \(\cD\) are positive and \(\eta\) is a smooth and rapidly decaying function of the Schwarz space \(\cS(\R^n)\). The generating function is chosen so that \(\cB_n \eta\) can be computed  analytically or efficiently numerically. If the generating function \(\eta\) satisfies the moment condition of order \(N\)
\begin{equation}\label{moment}
\int_{\R^n} \bx^\alpha \eta(\bx)d\bx=\delta_{0,\alpha},\qquad 0\leq |\alpha|<N \, ,
\end{equation}
then (\cite[p.21]{MSbook})
\begin{equation*}
|f(\bx)-f_{h,\cD}(\bx)|\leq c (\sqrt{\cD}h)^N\Vert\nabla_N f\Vert_{L_\infty} +\sum_{k=0}^{N-1}\e_k(h\sqrt{\cD})^k
|\nabla_k f(\bx)|
\end{equation*}
with 
\[
0<\e_k\leq \sum_{\bm\in\Z^n\setminus\{0\}}|\nabla_k \cF \h (\sqrt{\cD}\bm)|;\quad
\lim_{\cD\to \infty} \sum_{\bm\in\Z^n\setminus\{0\}} |\nabla_k \cF \h (\sqrt{\cD}\bm)|=0
\]
and
\[
|\nabla_k f(\bx)|=\sum_{|\alpha|=k}\frac{|\partial^\alpha f(\bx)|}{\alpha!}\,.
\]
Here $\cF \h$ denotes the Fourier transform of $\h$
\[
\cF\h(\by)=\itg_{\R^n} \h(\bx) \ee^{-2 i \pi \langle \bx,\by\rangle}d\bx.
\]
Hence, for any {\it saturation error} $\e>0$, one can fix the parameter $\cD>0$ so that 
\begin{equation*}\label{errq}
|f(\bx)-f_{h,\cD}(\bx)|=\cO((\sqrt{\cD}h)^N+\e)\Vert f\Vert_{W^N_\infty} \, ,
\end{equation*}
where $W^N_\infty=W^N_\infty(\R^n)$ denotes the Sobolev space of $L_\infty$-functions whose generalized derivatives up to the order $N$ also belong to $L_\infty$.
Then the linear combination 
\begin{equation}\label{cub1}
\cB_n f_{h,\cD}(\bx)= \frac{h^4}{\cD^{n/2-2}}\!\!\!\sum_{\bm\in\Z^n}f(h \bm)\cB_n\eta\left(\frac{\bx-h \bm}{h\sqrt{\cD}}\right)\end{equation}
gives rise to a new class of semi-analytic cubature formulas
with the property that, for any prescribed accuracy $\e>0$, one can fix the parameter $\cD$ so that  \eqref{cub1} differs in some uniform or $L_p$-norm from the integral \eqref{biharm} by 
\begin{equation*}\label{errror2}
\cO((\sqrt{\cD}h)^N+(\sqrt{\cD}h)^4\e)\qquad {\rm as}\qquad h\to0 \, ,
\end{equation*}
where $N$ is determined by $\eqref{moment}$. 
Estimates of the cubature error for general generating functions are proved  in Section \ref{sec2}.\\
Therefore, to construct cubature formulas for \eqref{biharm} it remains to compute 
the integral \(\cB_n \eta\). 
This can be taken analytically or transformed to a simple one-dimensional  integral. If we choose the generating functions 
\begin{equation*}
\eta(\bx)=\pi^{-n/2} L_{M-1}^{(n/2)}(|\bx|^2)\ee^{-|\bx|^2}
\end{equation*}
with the generalized Laguerre polynomials \(L_M^{(\gamma)}\) then \({\cB}_n \eta\)  can be taken analytically.  The function \(\eta\) satisfies the moment conditions \eqref{moment} with \(N=2M\) and \eqref{cub1} gives rise to semi-analytic cubature formulas for \(\cB_n f\) of order \(h^{2M}\) modulo the saturation error. In Section \ref{sec3} we describe these formulas when \(M=1\) that is  for the exponential \(\pi^{n/2}\ee^{-|\bx|^2}\) and, in Section \ref{sec4}, when \(M>1\).

If the generating function is the tensor product of one-dimensional functions of the form
\[
\eta(\bx)=\prod_{j=1}^n \pi^{-1/2} L_{M-1}^{(1/2)}(x_j^2)\ee^{-x_j^2}\, ,
\]
each of them satisfying the moment conditions \eqref{moment} of order \(2M\), then \({\cB}_n \eta\) is transformed to a one-dimensional integral with a separable integrand, i.e., a product of functions depending only on one of the variables. This is considered in Section \ref{sec5} where we obtain, for example, the integral representation
\begin{equation*}
\cB_n( \ee^{-|\cdot|^2})(\bx)= \frac{1}{16}\itg_0^\infty \frac{\ee^{-|\bx|^2/(1+t)}}{(1+t)^{n/2}}\,t\,dt,\qquad n\geq 5.
\end{equation*}
These one-dimensional integrals with separable integrand  in combination with a quadrature rule lead to accurate separated representations of the potential acting on the generating function. 
In Section \ref{sec6} for functions \(f\) with separated representations, i.e., within a given accuracy they can be represented as a sum of products of univariate functions, we derive formulas which reduces the \(n\)-dimensional convolution \eqref{biharm} to one-dimensional discrete convolutions. 
Thus for the computation of \eqref{biharm} only one-dimensional operations are used. We derive formulas of an arbitrary order fast and accurate  in high dimensions. 
We provide results of numerical experiments which show that even for very high space dimensions the approximations preserve the predicted convergence order \(2,4,6\) and \(8\) of the cubature.

\section{Cubature error}\label{sec2}
\setcounter{equation}{0}

The estimate of the cubature error  
\[
\cB_n f_{h,\cD}(\bx)-\cB_n f(\bx)=\cB_n (f_{h,\cD}- f)(\bx)
\]
for the biharmonic potential \eqref{biharm} is a consequence of 
the structure of the quasi-interpolation error, which is  proved
in general form  in  \cite[Thm 2.28]{MSbook}.
Suppose that $f$ has generalized derivatives of order $N$.
Using Taylor expansions of $f$ for the nodes $h\bm$, $\bm \in {\Z}^n$,
and Poisson's summation formula the quasi-interpolant can be written as
\begin{align}\label{repre}
f_{h,\cD}(\bx)= &(-h\sqrt{\cD})^N f_N(\bx)
+\sum _{ |\boldsymbol{\alpha}|  = 0} ^{N-1}
\frac{(h \sqrt{\cD})^{|\boldsymbol{\alpha}|}}{\boldsymbol{\alpha} !(2 \pi i)^{|{\boldsymbol{\alpha}}|}} 
\, \partial ^{\boldsymbol{\alpha}} f(\bx)  \,    \sigma_{\boldsymbol{\alpha}}(\bx,\h,\cD)
\end{align}
    with the function 
\begin{align*} 
f_N(\bx)= \frac{1 }{{\cD}^{n/2}}
 \sum _{ |\boldsymbol{\alpha}|  = N} \frac{N}{\ba!} \!
 \sum_{\bm \in {\Z}^n} 
\Big(\frac{ \bx-h \bm }{h\sqrt{\cD}}\Big)^{\boldsymbol{\alpha}}  \!
\h \Big(
\frac{\bx\!-\!h\bm}{h\sqrt{\cD}} \Big)
     \itg _0 ^1 s^{N-1} \partial ^{\boldsymbol{\alpha}} f
     (s \bx + (1-s) h\bm)  \, ds \, ,
\end{align*}
containing  the remainder of the Taylor expansions,
and the fast oscillating functions
   \begin{align} \label{defsigma}
     \sigma_{\boldsymbol{\alpha}}(\bx,\h,\cD)
=\sum_{\boldsymbol{\nu} \in {\Z}^n} 
\partial^{\boldsymbol{\alpha}}\cF \eta(\sqrt{\cD} {\boldsymbol{\nu}})
\, \ee^{ \frac{2 \pi i}{h} \langle {\mathbf x}, {\boldsymbol{\nu}}\rangle } \, .
   \end{align}
It follows from \eqref{defsigma} that due to the moment condition \eqref{moment} the second sum in \eqref{repre}
transforms to
   \begin{align*} 
f(\bx)+\sum _{ |\boldsymbol{\alpha}|  = 0} ^{N-1}
\frac{( h \sqrt{\cD})^{|\boldsymbol{\alpha}|}}{\boldsymbol{\alpha} !(2 \pi i)^{|{\boldsymbol{\alpha}}|}} 
\, \partial ^{\boldsymbol{\alpha}} f(\bx) \, \eps_{\boldsymbol{\alpha}}(\bx,\h,\cD) \, ,
   \end{align*}
where we denote 
   \begin{align*} 
\eps_{\boldsymbol{\alpha}}(\bx,\h,\cD) = \sum_{{\boldsymbol{\nu} \in {\Z}^n\setminus 0} } 
\partial^{\boldsymbol{\alpha}}\cF \eta(\sqrt{\cD} {\boldsymbol{\nu}})
\, \ee^{ \frac{2 \pi i}{h} \langle {\mathbf x}, {\boldsymbol{\nu}}\rangle } 
= \sigma_{\boldsymbol{\alpha}}(\bx,\h,\cD) - \delta_{0|\boldsymbol{\alpha}|}
\, .
    \end{align*}
     We denote by \( W_p^N= W_p^N(\R^n)\) the Sobolev space of \(L_p=L_p(\R^n)\) functions whose generalized 
derivatives up to order \(N\) belong to \(L_p\), with the norm
   \[
\Vert f\Vert_{W^N_p}=\sum_{l=0}^N |f|_{W^l_p},\quad |f|_{W^l_p}=\sum_{|\a|=l}\Vert \partial^\a f\Vert_{L_p}\,.
\]
If  $f \in W_p^N$ with $N>n/p$, $1 \le p \le \infty$, then $f_N$ can be estimated by
   \begin{equation*} %\label{ch1rem}
\|f_N\|_{L_p} \le C_N   |f|_{W_p^N},\qquad |f|_{W_p^N}=\sum _{ |\boldsymbol{\alpha}|  = N} \|\partial^{\boldsymbol{\alpha}}f\|_{L_p} \, ,
   \end{equation*}
with a constant $C_N$ depending only on $\h$, $n$, and $p$. 
Hence \eqref{repre} leads to the representation of the quasi-interpolation error
   \begin{align} \label{fullexp}
f_{h,\cD}(\bx)-f(\bx) 
=(-h\sqrt{\cD})^N 
f_N(\bx)
+\sum _{ |\boldsymbol{\alpha}|  = 0} ^{N-1}
\frac{(h \sqrt{\cD})^{|\boldsymbol{\alpha}|}}{\boldsymbol{\alpha} !(2 \pi i)^{|{\boldsymbol{\alpha}}|}} 
\, \partial ^{\boldsymbol{\alpha}} f(\bx) \eps_{\boldsymbol{\alpha}} (\bx,\h,\cD) ,
    \end{align}
which implies in particular the error estimate in $L_p$
   \begin{align} \label{estilp}
 \|f-f_{h,\cD}\|_{L_p} \le C_N (h\sqrt{\cD})^N|f|_{W_p^N}
+\sum _{ k = 0} ^{N-1} \frac{(h \sqrt{\cD})^k}{(2 \pi)^k} \sum _{ |\boldsymbol{\alpha}|  = k} 
\frac{\| \epsilon_{\boldsymbol{\alpha}}(\cdot,\h,\cD)\|_{L_\infty}
\| \partial ^{\boldsymbol{\alpha}} f\|_{L_p}}{\boldsymbol{\alpha}! } \, .
    \end{align}

Thus the quasi-interpolation error consists
of a term ensuring $\cO(h^N)$-convergence and
of the so-called saturation error, which, in general, does not converge to zero as $h \to 0$.
However,  due to the fast decay of $\partial^{\boldsymbol{\alpha}}\cF \eta$,  one can choose
$\cD$ large enough to ensure that 
   \begin{equation} \label{esteps}
\| \epsilon_{\boldsymbol{\alpha}}(\cdot,\cD,\h)\|_{L_\infty} \le 
\sum_{\boldsymbol{\nu} \in {\Z}^n\setminus \boldsymbol{0}}
|\partial^{\boldsymbol{\alpha}}\cF \eta(\sqrt{\cD} {\boldsymbol{\nu}})|<\e
    \end{equation}
for given small $\e > 0$.

From Sobolev's theorem we have that for \(n\geq 5\), \(1<p<n/4\), and \(q=np/(n-4p)\) the integral \eqref{biharm} converges absolutely for almost every \(\bx\) and  the operator \(\cB_n\) is a bounded mapping from \(L_p\) into \(L_q\) (cf. \cite[p. 119]{stein}). Hence
\begin{equation}\label{BLq}
\Vert\cB_n f_{h,\cD}-\cB_n f\Vert_{L_q}\leq A_{p,q}\Vert f_{h,\cD}- f\Vert_{L_p} \, ,
\end{equation}
where \(A_{p,q}\) denotes the norm of \(\cB_n:L_p\to L_q\). Then, from \eqref{estilp} and \eqref{esteps}, 

\begin{thm} \label{thm2.1}
Let \(n\geq 5\), \(1<p<n/4\), \(q=np/(n-4p)\) and \(f\in W^N_p\) with \(N>n/p\). Then, for any \(\eps>0\) there exists \(\cD>0\) such that
\begin{equation}\label{estrough}
\Vert\cB_n f_{h,\cD}-\cB_n f\Vert_{L_q}\leq A_{p,q}\Big(
C_N (h\sqrt{\cD})^N|f|_{W_p^N}
+\eps \sum _{ k = 0} ^{N-1} \frac{(h \sqrt{\cD})^k}{(2 \pi)^k}\|\nabla_k f\|_{L_{p}}\Big) \, .
\end{equation}
We used the notation
\begin{align*}
\|\nabla_k f\|_{L_{p}} = 
 \sum _{ |\boldsymbol{\alpha}|  = k} 
 \frac{\| \partial ^{\boldsymbol{\alpha}} f\|_{L_{p}}}{\boldsymbol{\alpha}! } \, .
\end{align*}
\end{thm}
It turns out, that under the conditions of Theorem \ref{thm2.1} the cubature formula \(\cB_n f_{h,\cD}\) 
converges to \(\cB_n f\). % with the error \(\cO(h^N+\eps h^4)\).
Since the biharmonic potential is a smoothing operator
and by \eqref{fullexp} the saturation error of the quasi-interpolant is a small, fast oscillating function, 
estimate \eqref{estrough} can be sharpened to the form that
\(\cB_n f_{h,\cD}\) approximates \(\cB_n f\) with the error \(\cO(h^N+\eps h^4)\).

We denote by \(H^s_p=H^s_p(\R^n)\) the Bessel potential space defined as the closure of compactly supported smooth functions with respect to the norm
\begin{equation*}\label{normHsp}
\Vert u\Vert_{H^s_p}=\Vert \cF^{-1}((1+4\pi^2|\cdot|^2)^{s/2} \cF u)\Vert_{L_p} =\Vert(1-\Delta)^{s/2}u\Vert_{L_p}\,.
\end{equation*}
We shall use the  error  estimate for the quasi-interpolant \eqref{fhD} in  the spaces \(H^s_p\)  
obtained in \cite[p.83]{MSbook}  which yields, in the case \(s=-4\), the following result.
\begin{thm}\label{estquasiint}\cite[p.83]{MSbook}
Suppose that \(\eta\in \cS(\R^n)\) satisfies the moments conditions  \eqref{moment} of order \(N\). Then, for any \(f\in H^L_p\), \(1<p<\infty\), \(L\geq N\geq 4\) with \(L> n/p\), there exist constants \(c_\eta\) and \(c_p\), not depending on \(f\) and \(h\) such that \(f_{h,\cD}\) defined in \eqref{fhD} satisfies
\begin{equation}\label{estH}
\begin{aligned}
\Vert f-f_{h,\cD}\Vert_{H^{-4}_p}\leq
c_\eta (h \sqrt{\cD})^N \Vert f \Vert_{H^L_p}+c_p h^4 \sum_{k=0}^{N-5} \frac{(h\sqrt{\cD})^k}{(2\pi)^{k+4}}\eps_k(\cD)
\sum_{|\a|=k}\Vert\partial^\a f\Vert_{H^4_p}
\end{aligned}
\end{equation}
with the numbers
\[
\eps_k(\cD)=\max_{|\a|=k} \sum_{\boldsymbol{\nu} \in {\Z}^n\setminus 0} 
\left|\partial^{\boldsymbol{\alpha}}\cF \eta(\sqrt{\cD} {\boldsymbol{\nu}})\right|\,.
\]
\end{thm}
Theorem \ref{estquasiint} leads to the following error estimate for the quasi-interpolation procedure.
\begin{thm}
Suppose that \(\eta\in \cS(\R^n)\) satisfies the moments conditions  \eqref{moment} of order \(N\).
Let \(n\geq 5\), \(1<p<n/4\), \(q=np/(n-4p)\) and \(f\in W^L_p\) with \(L\geq N\geq 4\) and \(L>n/p\). Then there exist constants \(c_\eta\) , \(c_p\) and \(c_q\), not depending on \(f, h, \cD\), such that
\[
\begin{aligned}
\Vert \cB_n f-\cB_n f_{h,\cD}\Vert_{L_q} \leq 
c_\eta (h \sqrt{\cD})^N&\Vert f \Vert_{W^L_p}+\\
h^4 \sum_{k=0}^{N-5} (h\sqrt{\cD})^k \frac{\eps_k(\cD)}{(2\pi)^{k+2}}& \sum_{l=0}^4\left(
A_{p,q} c_p |f|_{W^{l+k}_p}+c_q |f|_{W^{l+k}_q}\right)\,.
\end{aligned}
\]
\end{thm}
\begin{proof}
Since
\[
\cB_n f-\cB_n f_{h,\cD}=\cB_n (f-f_{h,\cD})
\]
we obtain, keeping in mind \eqref{BLq},
\[
\begin{aligned}
\Vert \cB_n (f-f_{h,\cD})\Vert_{L_q} &=
\Vert \cB_n (I-\Delta\Delta)  (I-\Delta\Delta)^{-1} (f-f_{h,\cD})\Vert_{L_q}\\
&\leq
\Vert \cB_n   (I-\Delta\Delta)^{-1} (f-f_{h,\cD})\Vert_{L_q}+\Vert \cB_n \Delta\Delta  (I-\Delta\Delta)^{-1} (f-f_{h,\cD})\Vert_{L_q}\\
&\leq A_{p,q}\Vert   (I-\Delta\Delta)^{-1} (f-f_{h,\cD})\Vert_{L_p}+\Vert  (I-\Delta\Delta)^{-1} (f-f_{h,\cD})\Vert_{L_q} \, .
\end{aligned}
\]
Since \((1+4\pi^2 |\xi|^2)^2\) can be bounded from above and from below by \((1+16\pi^4 |\xi|^4)\), the norm in \(H^{-4}_p\) is equivalent to 
\[
\Vert \cF^{-1}((1+16\pi^4|\cdot|^4)^{-1} \cF u)\Vert_{L_p} =\Vert(1-\Delta\Delta)^{-1}u\Vert_{L_p}\quad 
\]
and we deduce that
\begin{equation*} 
\begin{aligned}
\Vert \cB_n (f-f_{h,\cD})\Vert_{L_q} \leq  A_{p,q} \Vert f-f_{h,\cD}\Vert_{H_p^{-4}}+
\Vert f-f_{h,\cD}\Vert_{H_q^{-4}} \, .
\end{aligned}
\end{equation*}
The condition \(p<n/4\) ensures that \(W_p^L\) is continuously embedded in \(W_q^{L-4}\) (\cite[p.124]{stein}). Hence by  the  estimate  \eqref{estH} the assertion follows immediately.
\end{proof}

\section{Action on Gaussians
}\label{sec3}
\setcounter{equation}{0}

Consider the generating function \(\eta_2(\bx)=\pi^{-n/2}\ee^{-|\bx|^2}\). The moment conditions \eqref{moment} are fulfilled with \(N=2\). If we replace $f$ in \eqref{biharm} by
\begin{equation*}
f_{h,\cD}(\bx)=(\pi \cD)^{-n/2} \sum_{\bm\in\Z^n} f(h \bm) \ee^{-|\bx-h\bm|^2/(h^2\cD)} \, , 
\end{equation*}
then we obtain a cubature for \eqref{biharm}
\begin{equation}\label{biharm2}
\cB_n f_{h,\cD}(\bx)=\frac{(h\sqrt{\cD})^{4}}{(\pi \cD)^{n/2}} \sum_{\bm\in\Z^n} f(h\bm)\Phi_2\left(\frac{\bx-h\bm}{h\sqrt{\cD}}\right)
\end{equation}
with
\[
\Phi_2(\bx):=\cB_n( \ee^{-|\cdot|^2})(\bx).
\]
\begin{thm}%\label{1F1}
Let  $n\geq 3, n\neq 4$. The biharmonic potential acting on the Gaussian allows the following  representation
\begin{equation}\label{rapr}
\cB_n( \ee^{-|\cdot|^2})(\bx)=\frac{1}{4(n-2)(n-4) } 
 {_1F_1}(\frac {n-4}{2},\frac n2,-|\bx|^2) \, ,
\end{equation}
where \(_1F_1\) denotes the Kummer or confluent hypergeometric  function.
\end{thm}
\begin{proof}
The cubature of the \(3\)-dimensional biharmonic potential \(\cB_3 f\)
is considered in \cite[p.119]{MSbook}. To determine the action of the biharmonic potential on the Gaussian the general formula 
 \begin{equation} \label{congaus}
(Q\ast\ee^{-|\cdot|^2})(\bx)=
\frac{2\, \pi^{n/2}\,
\ee^{-|\bx|^{2}}}{|\bx|^{n/2-1}}
\int_{0}^{\infty}Q(r)  \ee^{-r^{2}}\, I_{n/2-1}(2 r |\bx|) \, r^{n/2}dr \,
     \end{equation}
with the  modified Bessel functions of the first kind $I_n$ (\cite[p.374]{AS}) and  \(Q(r)=-r/(8\pi)\) is used. Then  \eqref{congaus} gives
\begin{equation}\label{dim3M1}
\cB_3(\ee^{-|\cdot|^2})(\bx)=- \frac{\ee^{-|\bx|^2}}{8}-\frac{\sqrt{\pi}}{16}\frac{\erf(|\bx|)}{|\bx|} (2|\bx|^2+1)\,.
\end{equation}
\eqref{dim3M1} can also be expressed by means of the confluent hypergeometric functions as
\begin{equation*}\label{dim3M1bis}
\cB_3(\ee^{-|\cdot|^2})(\bx)=-\frac{1}{4}  {_1F_1}(-\frac 12,\frac 32,-|\bx|^2)\,.
\end{equation*}
Let \(n\geq 5\). The convolution of two radial functions can be transformed to a one-dimensional integral by using the Fourier transforms of the radial functions. Indeed
 (cf. \cite[(2.15) p.22]{MSbook})
\begin{equation}\label{2.15}
\int_{\R^n}Q(|\bx-\by|) f(|\by|)d\by=
\frac{2\pi}{|\bx|^{n/2-1}}\itg_0^\infty \cF Q(r) \cF f(r) J_{n/2-1}(2\pi r |\bx|) r^{n/2}dr\,.
\end{equation} 
Since \(\cF(\ee^{-|\cdot|^2})(\bx)=\pi^{n/2}\ee^{-\pi^2 |\bx|^2}\) and 
\(
\cF(|\cdot|^{4-n})(\bx)= \frac{\pi^{n/2-4}}{\Gamma\left(\frac{n-4}{2}\right)} |\bx|^{-4}
\)
(cf. \cite[p.156]{neri}) we have from \eqref{2.15} that
\[
\begin{aligned}
\cB_n( \ee^{-|\cdot|^2})(\bx)=
&\frac{\pi^{n/2-3}}{8 |\bx|^{n/2-1}} \itg_0^\infty \ee^{-\pi^2r^2} J_{n/2-1}(2 \pi r |\bx|) r^{n/2-4}dr \, ,
\end{aligned}
\]
where we used the relation
\[
\Gamma\left(\frac n2\right)=\frac{(n-2)(n-4)}{4}\Gamma\left(\frac{n-4}{2}\right)\,.
\]
This integral can be expressed by means of the Kummer or confluent hypergeometric function
\(
_1F_1(a,c,z)
\)  (cf. \cite[(8.6.14)]{MOT}). \eqref{rapr} follows.
\end{proof}
In particular, for \(n=5\), \eqref{rapr} gives
\[
\cB_5( \ee^{-|\cdot|^2})(\bx)=\frac{1}{16}\left(\frac{\ee^{-|\bx|^2}}{|\bx|^2}+{\sqrt{\pi}}\frac{\erf(|\bx|)}{2|\bx|^3}(2|\bx|^2-1)\right)
\]
and, for  \(n=6\), (cf. \cite[13.6]{AS}).
\[
\cB_6( \ee^{-|\cdot|^2})(\bx)= \frac{\ee^{-|\bx|^2}-1+|\bx|^2}{16 |\bx|^4}\, .
\]
In dimension \(n=4\) the biharmonic potential has the form
\begin{equation*}\label{biharm4}
\cB_4 f(\bx)=-\frac{1}{4\pi^2} \int_{\R^4}\ln |\bx-\by|f(\by)d\by
\end{equation*}
and the following representation formula holds.
\begin{thm} The biharmonic potential acting on the Gaussian function admits the representation
\begin{equation}\label{biharm4exp}
\cB_4(\ee^{-|\cdot|^2})(\bx)=\frac{1}{16}\left(\frac{\ee^{-|\bx|^2}-1}{|\bx|^2}
-2\log|\bx|-E_1(|\bx|^2)
\right)
\end{equation}
where \(E_1(r)\) is the exponential integral 
\begin{equation*}\label{expint}
E_1(r)=\int_r^\infty\frac{\ee^{-t}}{t}dt\,.
\end{equation*}
\end{thm}
\begin{proof}  We use that the radial function \(g(r)=\Phi_2(\bx)\), \(r=|\bx|\),  is  solution of the differential equation
\[\frac{1}{r^{3}}\frac{d}{dr}\left(r^{3} \frac{d}{dr}  \left(\frac{1}{r^{3}} \frac{d}{d r}\left( r^{3}\frac{d}{dr} \right)\right) \right)g(r)= \ee^{-r^2}, \quad r>0 \, ,
\]
satisfying the conditions
\[
g(r)\approx -\frac{1}{8}\log r \quad as \quad r\to \infty \, ,
\]
\[
g(0)=-\frac{1}{4} \int_0^\infty r^3 \log(r) \ee^{-r^2}dr=
\frac{\gamma-1}{16},\quad g'(0)=0
\]
with the Euler constant \(\gamma\).  Denote by \(\cL_n\) the inverse of the Laplace operator, the harmonic potential
\[
(\cL_n f)(\bx)=-\frac{\Gamma(n/2-1)}{4\pi^{n/2}}\int_{\R^n}\frac{f(y)}{|\bx-\by|^{n-2}}d\by \, ,
\]
which provides the unique solution of the Poisson equation
\[
\Delta u=f\quad {\rm in}\quad \R^n,\qquad |u(\bx)|\leq c|\bx|^{n-2}\quad |\bx|\to \infty \, .
\]
Hence we have
\begin{equation*}\label{form}
 \frac{1}{r^{3}} \frac{d}{d r}\left( r^{3}\frac{d}{dr} \right)g=\cL_4(\ee^{-|\cdot|^2})(\bx)\,.
\end{equation*}
From the relation (cf. \cite[p.75]{MSbook})
\[
\cL_4(\ee^{-|\cdot|^2})(\bx) =\frac{\ee^{-|\bx|^2}-1}{4 |\bx|^2}
\]
we deduce that \(g\) solves
\[
g''(r)+3 \frac{ g'(r)}{r}=\frac{\ee^{-r^2}-1}{4 r^2},\quad
 g(0)=
\frac{\gamma-1}{16},\quad g'(0)=0.
\]
We obtain
\[
g'(r)=\frac{1}{8}\left(\frac{1}{r^3}-\frac{\ee^{-r^2}}{r^3}-\frac{1}{r} \right)
\]
and, finally
\[
\begin{aligned}
g(r)&=\frac{1}{8}\itg_0^r \left(\frac{1}{s^3}-\frac{\ee^{-s^2}}{s^3}-\frac{1}{s} \right)ds+g(0)=\frac{1}{16}\itg_0^{r^2} \left(\frac{1}{t}-\frac{\ee^{-t}}{t}-1 \right)\frac{dt}{t}\\&+g(0)
=\frac{1}{16}\left(\frac{\ee^{-r^2}-1}{r^2}-E_1(r^2)-2 \log r
\right)\,
\end{aligned}
\]
with the exponential integral \(E_1\) (\cite[5.1.1]{AS}).
\end{proof}
\eqref{biharm2}, together with \eqref{rapr} and \eqref{biharm4exp}, gives rise to second order semi-analytic cubature formulas for the biharmonic operator in any dimension \(n\geq 3\).

\section{Action on higher-order basis functions}\label{sec4}
\setcounter{equation}{0}

Now we consider the biharmonic potential
\[
\Phi_{2M}(\bx):=\cB_n(L_{M-1}^{(n/2)}(|\cdot|^2)\ee^{-|\cdot|^2})(\bx)\]
of the radial function \(L_{M-1}^{(n/2)}(|\bx|^2)\ee^{-|\bx|^2}\) with  the generalized Laguerre polynomials 
\[
L_k^{(\gamma)}(y)=\frac{\ee^y y^{-\gamma}}{k!} \left( \frac{d}{dy}\right)^k (\ee^{-y} y^{k+\gamma}),\quad \gamma>-1\,.
\]
The radial functions
\begin{equation}\label{gen}
\eta_{2M} (\bx)=\pi^{-n/2} L_{M-1}^{(n/2)}(|\bx|^2)\ee^{-|\bx|^2}
\end{equation}
satisfy the moment conditions of order \(2M\) (\cite[p.56]{MSbook}) and give rise to approximation formulas of order \(2M\) modulo the saturation error. If we give an analytic  formula for \(\Phi_{2M}\), then we obtain the following  semi-analytic cubature for \eqref{biharm}
\begin{equation}\label{biharmN}
\cB_n f_{h,\cD}(\bx)=\frac{(h\sqrt{\cD})^{4}}{(\pi\cD)^{n/2}} \sum_{\bm\in\Z^n} f(h\bm)\Phi_{2M}\left(\frac{\bx-h\bm}{h\sqrt{\cD}}\right)
\end{equation}

\begin{thm} \label{prop4.1}
For \(M>1\) we have
\[\Phi_{2M}(\bx)=
\cB_n(\ee^{-|\cdot|^2})(\bx)+ \dfrac{1}{16|\bx|^{n-2}}\gamma(\dfrac{n}{2}-1,|\bx|^2) 
+\dfrac{\ee^{-|\bx|^2}}{16} \sum_{j=0}^{M-3} \dfrac{L_j^{(n/2-1)}(|\bx|^2)}{(j+1)(j+2)}
\]
with the lower incomplete Gamma function
\begin{equation*}\label{gamma}
\gamma(a,x)=\int_0^x t^{a-1} \ee^{-t}dt\,.
\end{equation*}
\end{thm}
\begin{proof}
We use the relation \cite[(3.18)]{MSbook}
\[
\begin{aligned}
\ L_{M-1}^{(n/2)}(|\bx|^2)\ee^{-|\bx|^2}=\sum_{j=0}^{M-1} \frac{(-1)^j}{j! 4^j}\Delta^j \ee^{-|\bx|^2}
= \ee^{-|\bx|^2} -\frac{1
}{4}\Delta \ee^{-|\bx|^2}+\sum_{j=2}^{M-1} \frac{(-1)^j}{j! 4^j}\Delta^j \ee^{-|\bx|^2}\,.
\end{aligned}
\]
%We consider 
Let \(\cL_n\) be the inverse of the Laplace operator, that is \(\cL_n \Delta=I\). Then \(\cB_n(\Delta \ee^{-|\cdot|^2})=\cL_n( \ee^{-|\cdot|^2}\))
and we have
\[
\begin{aligned}
\cB_n(\ L_{M-1}^{(n/2)}(|\cdot|^2)\ee^{-|\cdot|^2})(\bx)=
\cB_n(\ee^{-|\cdot|^2})(\bx)-\frac {1}{4} \cL_n(\ee^{-|\cdot|^2})(\bx)+\sum_{j=2}^{M-1} \frac{(-1)^j}{j! 4^j}\Delta^{j-2} \ee^{-|\bx|^2} \, .
\end{aligned}
\]
From the relations (\cite[p.75]{MSbook})
\[
\begin{aligned}
 \cL_n(\ee^{-|\cdot|^2})(\bx)&=-\frac{1}{4|\bx|^{n-2}}\gamma(\frac{n}{2}-1,|\bx|^2),\quad n\geq 3 \, ,
\end{aligned}
\]
with the lower incomplete Gamma function \(\gamma(a,\bx)\)
and (\cite[(4.24)]{MSbook})
\[
\sum_{j=2}^{M-1} \frac{(-1)^j}{j! 4^j}\Delta^{j-2} \ee^{-|\bx|^2}=
\frac{1}{16}\sum_{s=0}^{M-3} \frac{(-1)^s}{4^s (s+2)!}\Delta^s \ee^{-|\bx|^2}
=\frac{\ee^{-|\bx|^2}}{16} \sum_{j=0}^{M-3} \frac{L_j^{(n/2-1)}(|\bx|^2)}{(j+1)(j+2)}
\]
we conclude the proof.
\end{proof}
From  the relation (\cite[6.5.16]{AS})
\[
\gamma(\frac{1}{2} , x^2)=\sqrt{\pi} \erf(x)
\] 
and the recurrence relation (\cite[6.5.22]{AS})
\[\gamma(a+1,x^2)=a \gamma(a,x)-\ee^{-x} x^2\]
we see that  in the case of odd space dimension the biharmonic potential of the Gaussian is expressed using the error function \(\erf\). In the case of even space dimension the biharmonic potential of the Gaussian is expressed by elementary functions  since ({\cite[6.5.13]{AS}}) 
\[
\gamma(k,x)=(k-1)! \left(1-\ee^{-x} \sum_{j=0}^{k-1} \frac{x^j}{j!}\right),\qquad k\in \N.
\]

In particular we have

\begin{align*}\label{dim3M}
\Phi_{2M}(\bx)&=
- \frac{\ee^{-|\bx|^2}}{8}-\frac{\sqrt{\pi}|\bx|}{8}{\erf(|\bx|)}
+\frac{\ee^{-|\bx|^2}}{16} \sum_{j=0}^{M-3} \frac{L_j^{(1/2)}(|\bx|^2)}{(j+1)(j+2)}\quad &{\rm for}\quad n=3,
\\
\Phi_{2M}(\bx)&=-\frac{\log|\bx|}{8}-\frac{E_1(|\bx|^2)}{16}+
\frac{\ee^{-|\bx|^2}}{16} \sum_{j=0}^{M-3} \frac{L_j^{(1)}(|\bx|^2)}{(j+1)(j+2)}\quad &{\rm for}\quad n=4,
\\
\Phi_{2M}(\bx)&=
\frac{\sqrt{\pi}}{16} \frac{\erf(|\bx|)}{|\bx|} +
\dfrac{\ee^{-|\bx|^2}}{16} \sum_{j=0}^{M-3} \dfrac{L_j^{(3/2)}(|\bx|^2)}{(j+1)(j+2)}\quad &{\rm for}\quad n=5,
\\
\Phi_{2M}(\bx)&=
\frac{1-\ee^{-|\bx|^2}}{16 |\bx|^2} 
+\dfrac{\ee^{-|\bx|^2}}{16} \sum_{j=0}^{M-3} \dfrac{L_j^{(2)}(|\bx|^2)}{(j+1)(j+2)}
\quad &{\rm for}\quad n=6.
\end{align*}

Theorem \ref{prop4.1} shows that \(\Phi_{2M+2}\) can be obtained from  \(\Phi_{2M}\) by adding some rapidly decaying terms.
We conclude that the approximation of the density \(f\) by the quasi-interpolant \eqref{fhD}  with the basis functions \eqref{gen} leads to the semi-analytic approximation of the biharmonic potential \eqref{biharmN}
and the corresponding analytic expression for \(\Phi_{2M}\) has to be used.

\section{Separated representation of the biharmonic potential acting on Gaussians}\label{sec5}
\setcounter{equation}{0}

In this section we take the tensor product generation function 
\begin{equation}
\begin{split}\label{basis}
&\h_{2M}(\bx)=\prod_{j=1}^n\widetilde{\h}_{2M}(x_j);
\quad
\widetilde{\h}_{2M}(x_j)=\frac{(-1)^{M-1}}{2^{2M-1}\sqrt{\pi} (M-1)!}\frac{H_{2M-1}(x_j) \ee^{-x_j^2}}{x_j} \, ,
\end{split}
\end{equation}
which satisfies the moment conditions of order \(2M\), where $H_k$ are the Hermite polynomials
\begin{equation*}\label{hermite}
H_k(x)=(-1)^k \ee^{x^2} \left( \frac{d}{dx}\right)^k \ee^{-x^2}\,.
\end{equation*}
The \(n\)-dimensional potential \(\cB_n\) applied to the basis function \(\eta_{2M}\) can be transformed to a one-dimensional integral with separable integrand, i.e., a product of functions depending only on one of the variables. In Section \ref{sec6}  we will show how these one-dimensional integrals, in combination with a quadrature rule, lead to accurate separated representations of the potential acting on the generating function. Hence, for functions \(f\) with separated representations, we derive fast formulas which reduces the \(n\)-dimensional convolution \eqref{biharm} by one-dimensional discrete convolutions. 

We start with second order approximations, i.e. \(M=1\).
\begin{thm} \label{thm5.1}
The biharmonic potential \(\cB_n(\ee^{-|\cdot|^2})\) admits the following one-dimensional integral
representation
\begin{equation}\label{rapr3}
\cB_3( \ee^{-|\cdot|^2})(\bx)= -\frac{1}{8}\itg_0^\infty \ee^{-|\bx|^2/(1+t)} \left(
\frac{1}{(1+t)^{3/2}}+\frac{t |\bx|^2}{(1+t)^{5/2}}
\right) dt\,,
\end{equation}
\begin{equation}\label{rapr2}
\cB_n( \ee^{-|\cdot|^2})(\bx)= \frac{1}{16}\itg_0^\infty \frac{\ee^{-|\bx|^2/(1+t)}}{(1+t)^{n/2}}\,t\,dt,\quad n\geq 5\,.
\end{equation}
\end{thm}
\begin{proof}
We use the integral formula \cite[13.2.1]{AS}
\begin{align}\label{int1F1}
_1F_1(a,c,z) = \frac{\Gamma(c)}{\Gamma(a)\Gamma(c-a)} \itg_0^1 \ee^{\,z\tau} 
\tau^{a-1}(1-\tau)^{-a+c-1} d\tau, \quad \re(c)>\re(a)>0 \, .
\end{align}
Let \(n\geq 5\). With the substitution $\tau=1/(1+t)$ we get
\begin{align*}
 {_1F_1}(\frac {n-4}{2},\frac n2,-|\bx|^2)
&= \frac{\Gamma(\frac n2)}{\Gamma(\frac {n}{2}-2)}
\itg_0^\infty\ee^{-|\bx|^2/(1+t)} \frac{1}{(1+t)^{n/2-3}}\frac{t}{1+t}
\frac{dt}{(1+t)^2}\\
&=\frac{(n-2)(n-4)}{4}\itg_0^\infty  \frac{\ee^{-|\bx|^2/(1+t)}}{(1+t)^{n/2}}\, t \, dt \, .
\end{align*}
From \eqref{rapr} we get \eqref{rapr2}.
\\
From the recurrence relation (cf. \cite[13.4.4]{AS}) we have
\[
{_1 F_1(-\frac 12,\frac 32,-|\bx|^2)}={_1F_1}(\frac 12,\frac 32,-|\bx|^2)+\frac 23 |\bx|^2 {_1F_1}(\frac 12,\frac 52,-|\bx|^2).
\]
The integral formula \eqref{int1F1} gives
\[
{_1F_1}(\frac 12,\frac 32,-|\bx|^2)=\frac 12 \int_0^1 \frac{\ee^{- |\bx|^2 \tau}}{\tau^{1/2}}d\tau=
\frac 12 \int_0^\infty \frac{\ee^{-|\bx|^2/(1+t)}}{(1+t)^{3/2}}dt\,,
\]
\[
{_1F_1}(\frac 12, \frac 52, -|\bx|^2)=\frac{3}{4} \int_0^\infty \frac{ \ee^{-|\bx|^2/(1+t)}}{(1+t)^{5/2}}\, t \, dt \, .
\]
\eqref{rapr3} follows from \eqref{rapr}.
\end{proof}
In the next theorem we derive one-dimensional integral representations for the potential \(\cB_n\) acting on the basis functions \eqref{basis}.
\begin{thm} 
For \(M>1\) we have
\begin{multline}\label{dimM3bis} 
\cB_3(\eta_{2M})(\bx)=
-\frac{1}{8 \pi^{3/2}}\left(  \itg_0^\infty  \prod_{j=1}^3 \frac{\ee^{-x_j^2/(1+t)}}{\sqrt{1+t}} Q_M(x_j,t) dt
\right.\\
\left.
+\itg_0^\infty  \sum_{i=1}^3 \frac{\ee^{-x_i^2/(1+t)}}{\sqrt{1+t}}R_M(x_i,t) \prod_{\substack{j=1\\j\neq i}}^3 \frac{\ee^{-x_j^2/(1+t)}}{\sqrt{1+t}}  Q_M(x_j,t)  t\,dt\right)
\end{multline}
\begin{equation}\label{dimnMbis}
\cB_n(\eta_{2M})(\bx)=
\frac{1}{16 \pi^{n/2}}  \itg_0^\infty  \prod_{j=1}^n \frac{\ee^{-x_j^2/(1+t)}}{\sqrt{1+t}} Q_M(x_j,t) \, t\,dt\,,\qquad n\geq 5
\end{equation}
with \(\eta_{2M}\) in \eqref{basis} and
\[
Q_M(x,t)=  \sum_{k=0}^{M-1}  \frac{(-1)^k}{k!4^k}
\frac{1}{(1+t)^{k}}
H_{2k}\left(\frac{x}{\sqrt{1+t}}\right) \,;
\]
\begin{equation}\label{RM}
R_M(x,t)= \sum_{k=0}^{M-1} \frac{(-1)^k}{k!4^k}
\frac{1}{(1+t)^k} \cS_{2k}\left(\frac{x}{\sqrt{1+t}}\right)\,;
\end{equation}
\begin{equation}\label{Sk}
\cS_k(y)=y^2H_{k}(y) -2kyH_{k-1}(y)+k(k-1)H_{k-2}(y)\,.
\end{equation}
\(Q_M(x,t)\) and \(R_M(x,t)\) are polynomials in \(x\) whose coefficients depend on \(t\).
\end{thm}
\begin{proof}
To get a one-dimensional integral representation  of \(\cB_n(\prod_{j=1}^n \widetilde{\eta}_{2M})\)  we use the 
relation (\cite[p.55]{MSbook})
\[
\widetilde{ \h}_{2M}(y)=\frac{1}{\sqrt{\pi}} \sum_{k=0}^{M-1} \frac{(-1)^k}{k!4^k}
\frac{d^{2k}}{dy^{2k}}\ee^{-{y^2}} \, .
\]
Let \(n\geq 5\). The solution of the equation
\[
\Delta\Delta u=\prod_{j=1}^n \widetilde{\eta}_{2M}(x_j)
\]
is given by the integral
\[
\begin{aligned}
\frac{1}{16} \prod_{j=1}^n \frac{1}{\sqrt{\pi}} &\sum_{k=0}^{M-1} \frac{(-1)^k}{k!4^k}
\frac{d^{2k}}{d x_j^{2k}}  \itg_0^\infty \frac{\ee^{-x_j^2/(1+t)}}{(1+t)^{n/2}}\,t\,dt\\
&=
\frac{1}{16}  \itg_0^\infty \left( \prod_{j=1}^n \frac{1}{\sqrt{\pi}} \sum_{k=0}^{M-1} \frac{(-1)^k}{k!4^k}
\frac{d^{2k}}{d x_j^{2k}}  \ee^{-x_j^2/(1+t)}\right)\frac{t\,dt}{(1+t)^{n/2}}\\
&=\frac{1}{16}  \itg_0^\infty \left( \prod_{j=1}^n  \frac{1}{\sqrt{\pi}} \sum_{k=0}^{M-1}  \frac{(-1)^k}{k!4^k}
\frac{\ee^{-x_j^2/(1+t)}}{(1+t)^{k+1/2}}
H_{2k}\left(\frac{x_j}{\sqrt{1+t}}\right) \right) {t\,dt} \, ,
\end{aligned}
\]
that is \eqref{dimnMbis}.\\
Let \(n=3\). Keeping in mind \eqref{rapr3}, we get
\begin{multline}\notag
\cB_3(\eta_{2M} (\cdot))(\bx)= -\frac{1}{8\pi^{3/2}} \prod_{j=1}^3 \sum_{k=0}^{M-1} \frac{(-1)^k}{k!4^k}
\frac{d^{2k}}{d x_j^{2k}}  \itg_0^\infty 
\frac{ \ee^{-|\bx|^2/(1+t)} }{(1+t)^{3/2}}dt\\
-\frac{1}{8\pi^{3/2}} \prod_{j=1}^3 \sum_{k=0}^{M-1} \frac{(-1)^k}{k!4^k}
\frac{d^{2k}}{d x_j^{2k}}  \itg_0^\infty\frac{ |\bx|^2  \ee^{-|\bx|^2/(1+t)} }{(1+t)^{5/2}}tdt.
\end{multline}
The first term in the r.h.s is similar to that considered in the case \(n\geq 5\). 
\\ Concerning the second term, we have
\begin{eqnarray*}
\frac{d^{2k}}{d x^{2k}} \left(\frac{x^2}{1+t}  \ee^{-x^2/(1+t)} \right)=\frac{1}{(1+t)^k} 
\frac{d^{2k}}{d y^{2k}} \left(y^2  \ee^{-y^2} \right)_{y=x/\sqrt{1+t}}=
\frac{\ee^{-x^2/(1+t)}}{(1+t)^k} \cS_{2k}\left(\frac{x}{\sqrt{1+t}}\right)
\end{eqnarray*}
with \(\cS_k(y)\) in \eqref{Sk}.
Then
\[\sum_{k=0}^{M-1}   \frac{(-1)^k}{k!4^k}
\frac{d^{2k}}{d x^{2k}} \left(\frac{x^2}{1+t}  \ee^{-x^2/(1+t)}\right)=R_M(x,t)  \ee^{-x^2/(1+t)}
\]
with \(R_M(x,t)\) defined in \eqref{RM}. It follows that the second integral can be written as
\begin{align*}
\prod_{j=1}^3 &\sum_{k=0}^{M-1} \frac{(-1)^k}{k!4^k}
\frac{d^{2k}}{d x_j^{2k}}  \itg_0^\infty\frac{|\bx|^2  \ee^{-|\bx|^2/(1+t)} }{(1+t)^{5/2}}tdt\\
&=
 \itg_0^\infty \prod_{j=1}^3 \sum_{k=0}^{M-1}   \frac{(-1)^k}{k!4^k}
\frac{d^{2k}}{d x_j^{2k}} |\bx|^2  \ee^{-|\bx|^2/(1+t)} \frac{ tdt}{(1+t)^{5/2}}\\
&=\itg_0^\infty
\sum_{i=1}^3R_M(x_i,t) \prod_{\substack{j=1\\j\neq i}}^3 Q_M(x_j,t)
 \frac{\ee^{-|\bx|^2/(1+t)} tdt}{(1+t)^{3/2}} \, ,
\end{align*}
which leads to \eqref{dimM3bis}.
\end{proof}

The polynomials $Q_M(x,t)$ and \(R_M(x,t)\) 
for $M=1,2,3,4$ are given by
\begin{align*}
&Q_1(x,t)=1 \, , \qquad
Q_2(x,t)=-\frac{x^2}{(1+t)^2}+\frac{1}{2 (1+t)}+1 ,\\
&Q_3(x,t)=Q_2(x,t)+\frac{x^4}{2 (1+t)^4}-\frac{3 x^2}{2 (1+t)^3}+\frac{3}{8 (1+t)^2}\, ,\\
&Q_4(x,t)=Q_3(x,t)-\frac{x^6}{6 (1+t)^6}+\frac{5 x^4}{4 (1+t)^5}-\frac{15 x^2}{8
   (1+t)^4}+\frac{5}{16 (1+t)^3}\, ,\\
% \end{align*}
% \begin{align*}
&R_1(x,t)=\frac{x^2}{1+t} \, , \qquad
R_2(x,t)= -\frac{x^4}{(1+t)^3}+\frac{x^2}{1+t}+\frac{5 x^2}{2 (1+t)^2}-\frac{1}{2 (1+t)},\\
&R_3(x,t)=R_2(x,t)+\frac{x^6}{2 (1+t)^5}-\frac{7 x^4}{2 (1+t)^4}+\frac{39 x^2}{8
   (1+t)^3}-\frac{3}{4 (1+t)^2},\\
   &R_4(x,t)=R_3(x,t) -\frac{x^8}{6 (1+t)^7}+\frac{9 x^6}{4 (1+t)^6}-\frac{65 x^4}{8
   (1+t)^5}+\frac{125 x^2}{16 (1+t)^4}-\frac{15}{16 (1+t)^3} .
\end{align*}

\section{Implementation and numerical results} \label{sec6}
\setcounter{equation}{0}

In this section we consider the fast computation of the biharmonic potential based on \eqref{basis}.
From  \eqref{biharmN} and \eqref{dimnMbis} we derive the cubature formula
\begin{equation*}\label{fastcub}
\cB_{M,h}^{(n)} f (\bx)=\frac{(h\sqrt{\cD})^{4}}{16 (\pi \cD)^{n/2}} \sum_{\bm\in\Z^n} f(h\bm)
  \itg_0^\infty  \prod_{j=1}^n Q_M\left(\frac{x_j-hm_j}{h\sqrt{\cD}},t\right) \, t\,dt\,,\quad n\geq 5\,.
\end{equation*}
At the grid points \(h\bk=(hk_1,...,hk_n)\) we obtain
\begin{equation}\label{convsum}
\cB_{M,h}^{(n)} f(h\bk)=\frac{(h\sqrt{\cD})^{4}}{16} \sum_{\bm\in\Z^n} f(h\bm)
 a_{\bk-\bm}^{(M)}
\end{equation}
where
\begin{equation}\label{akM}
a_{\bk}^{(M)}= \frac{1}{ (\pi \cD)^{n/2}}\itg_0^\infty  \prod_{j=1}^n \ee^{-k_j^2/(\cD(1+t))} Q_M\left(\frac{k_j}{\sqrt{\cD}},t\right) \, t\,dt\,.
 \end{equation}
 The product structure of the integrand leads to new cubature formulas if the density \(f\) admits the so-called separated representation. The idea is the following.  If  \(f\) is given as product of univariate functions
 \[
 f(\bx)=\prod_{j=1}^n f_j(x_j)
 \]
 then  the values on the grid \(h\bk\) of the cubature formula can be written as
 \[
\cB_{M,h}^{(n)} f(h\bk)=\frac{(h\sqrt{\cD})^{4}}{16 (\pi \cD)^{n/2}}  \itg_0^\infty  \prod_{j=1}^n \sum_{m_j\in\Z} f_j(hm_j) \ee^{-(k_j-m_j)^2/(\cD(1+t))} Q_M\left(\frac{k_j-m_j}{\sqrt{\cD}},t\right) t\,dt\,.
 \]
 A  suitable quadrature of this integral with nodes \(\tau_s\) and quadrature weights \(\omega_s\) leads to 
 \[
\cB_{M,h}^{(n)} f(h\bk)\approx
 \frac{(h\sqrt{\cD})^{4}}{16(\pi \cD)^{n/2}}  \sum_s \omega_s \prod_{j=1}^n \sigma_j(k_j,\tau_s)
 \]
 with 
 \[
\sigma_j(k,\tau)=\sum_{m\in\Z} f_j(hm) \ee^{-(k-m)^2/(\cD(1+\tau))} Q_M\left(\frac{k-m}{\sqrt{\cD}},\tau\right) \tau.
 \]
Then the value of the integral operator on the grid \(h\bk\) can be obtained by computing one-dimensional sums, and therefore the computational complexity of the algorithm scales linearly in the physical dimension.\\

We use an efficient quadrature based on the classical trapezoidal rule, which is exponentially converging for rapidly decaying smooth functions on the real line. We make the substitutions
 \[
 t=\ee^\xi,\quad \xi=a (\sigma+\ee^\sigma),\quad \sigma=b (u-\ee^{-u})
 \]
 with positive constants \(a\) and \(b\) proposed in \cite{Mo} (see also \cite{LMS2, LMS3}). Then the integrals \eqref{akM} are transformed to integrals over \(\R\) with integrands decaying  doubly exponentially in \(u\). After the substitution we have
 \begin{align*} 
a_{\bk}^{(M)}= 
\frac{1}{ (\pi \cD)^{n/2}}\itg_{-\infty}^\infty  \prod_{j=1}^n \ee^{-k_j^2/(\cD(1+\Phi(u)))} Q_M\left(\frac{k_j}{\sqrt{\cD}},\Phi(u)\right) \, \Phi(u)\Phi'(u)du
\end{align*}
with the functions
  \begin{align*} 
  \Phi(u)={\rm exp}(ab\,(u-\ee^{-u})+a \,{\rm exp}(b(u-\ee^{-u})))  ) \, ,\\
  \Phi'(u)=\Phi(u) ab(1+\ee^{-u})(1+{\rm exp}(b(u-\ee^{-u})))\,.
 \end{align*}
Thus the trapezoidal rule of step \(\tau\) can provide very accurate approximations of the integral for a relatively small number of nodes \(\tau s\)
 \begin{align*} 
a_{\bk}^{(M)}\approx 
\frac{\tau}{ (\pi \cD)^{n/2}}\sum_s   \prod_{j=1}^n \ee^{-k_j^2/(\cD(1+\Phi(\tau s)))} Q_M\left(\frac{k_j}{\sqrt{\cD}},\Phi(\tau s)\right) \, \Phi(\tau s)\Phi'(\tau s)
\end{align*}
Assume  that \(f\), within a prescribed accuracy, can be represented as sum of products of one-dimensional functions
 \begin{equation}\label{sepf}
 f(\bx)=\sum_{p=1}^P \beta_p \prod_{j=1}^n f_j^{(p)}(x_j)+\cO(\epsilon)
 \end{equation}
 with suitable functions \(f_j^{(p)}\) chosen such that the separation rank \(P\) is small. We derive the approximation of the convolutional sum \eqref{convsum} using one-dimensional operations
 \begin{align} \notag
\cB_{M,h}^{(n)} f(h\bk)& \approx  \frac{(h\sqrt{\cD})^{4}}{16(\pi \cD)^{n/2}} \tau 
 \sum_{p=1}^P \beta_p \sum_s \Phi(\tau s)\Phi'(\tau s)\\\notag
& \times  \prod_{j=1}^n \sum_{m_j}f_j^{(p)}(hm_j) 
  \left(
   \ee^{-(k_j-m_j)^2/(\cD(1+\Phi(\tau s)))} Q_M\left(\frac{k_j-m_j}{\sqrt{\cD}},\Phi(\tau s)\right)
  \right)\,.
 \end{align} 
We provide  results of some experiments which show the accuracy and numerical order of the method. We compute the biharmonic potential of the density 
\begin{equation}\label{f(x)}
f(\bx)=4 \ee^{-|\bx|^2}(n (n+2)-4 (n+2) |\bx|^2 +4 |\bx|^4) \, ,
\end{equation}
which has exact values
\(
\cB_{n}f(\bx)=\ee^{-|\bx|^2}
\). In Table \ref{T1} we compare the exact values of \(\cB_n f\) and the approximate values \(\cB_{4,0.025}^{(n)} f\) at some grid points \((x_1,0,...,0)\in \R^n\) for space dimensions \(n=5,10,10^2,...,10^8\). 

\begin{table}[h]
\begin{footnotesize}
%\begin{scriptsize}
\begin{tabular}{r|c|cc|cc|cc} \hline
$n$ &&\multicolumn{2}{c|}{5}&\multicolumn{2}{c|}{10}&\multicolumn{2}{c}{100}\\ \hline
 $x_1$ & exact & abs. error &rel. error & abs. error & rel. error & abs. error & rel. error\\ \hline
 0 &0.100E+01 & 0.129E-09  &  0.129E-09 & 0.258E-09  &  0.258E-09 & 0.258E-08  &  0.258E-08 \\
 1 &0.368E+00 & 0.286E-10  &  0.777E-10& 0.760E-10  &  0.207E-09 & 0.930E-09  &  0.253E-08 \\
 2 &0.183E-01 & 0.171E-11  &  0.933E-10 & 0.404E-11  &  0.220E-09 & 0.465E-10  &  0.254E-08 \\
 3 & 0.123E-03& 0.112E-12  &  0.910E-09 & 0.943E-13  &  0.764E-09 & 0.381E-12  &  0.309E-08 \\
 4 & 0.113E-06 &0.435E-13  &  0.386E-06 & 0.948E-14  &  0.843E-07 & 0.973E-14  &  0.864E-07 \\
\hline 
\end{tabular}\\[1mm]
\begin{tabular}{r|c|cc|cc|cc} \hline
 $n$ &&\multicolumn{2}{c|}{$1\,000$ }&\multicolumn{2}{c|}{$10\,000$}&\multicolumn{2}{c}{$100\,000$ }\\ \hline
 $x_1$ & exact & abs. error &rel. error & abs. error & rel. error & abs. error & rel. error \\ \hline
 0 & 0.100E+01  & 0.258E-07  &  0.258E-07&  0.258E-06  &  0.258E-06 & 0.258E-05  &  0.258E-05   \\
 1 & 0.368E+00  &0.947E-08  &  0.257E-07&  0.948E-07  &  0.258E-06 & 0.949E-06  &  0.258E-05   \\
 2 & 0.183E-01  & 0.472E-09  &  0.257E-07&  0.472E-08  &  0.258E-06 & 0.472E-07  &  0.258E-05  \\
  3 & 0.123E-03 &  0.324E-11  &  0.263E-07&  0.319E-10  &  0.258E-06 & 0.318E-09  &  0.258E-05  \\
 4 & 0.113E-06  & 0.123E-13  &  0.110E-06&  0.385E-13  &  0.342E-06 &  0.300E-12  &  0.266E-05  \\
\hline 
\end{tabular}\\[1mm]
\begin{tabular}{r|c|cc|cc|cc} \hline
  $n$ &&\multicolumn{2}{c|}{$1\,000\,000$ }&\multicolumn{2}{c|}{$10\,000\,000$}&\multicolumn{2}{c}{$100\,000\,000$ }\\ \hline
 $x_1$ & exact & abs. error &rel. error & abs. error & rel. error & abs. error & rel. error\\ \hline
 0 & 0.100E+01  &  0.258E-04  &  0.258E-04&0.258E-03  &  0.258E-03 & 0.258E-02  &  0.258E-02  \\
 1 &0.368E+00  &  0.949E-05  &  0.258E-04&0.948E-04  &  0.258E-03 & 0.947E-03  &  0.258E-02  \\
 2 & 0.183E-01  &  0.472E-06  &  0.258E-04&0.472E-05  &  0.258E-03& 0.472E-04  &  0.258E-02 \\
  3 &0.123E-03  &  0.318E-08  &  0.258E-04&0.318E-07  &  0.258E-03& 0.318E-06  &  0.258E-02  \\
 4 &0.113E-06  &  0.291E-11  &  0.259E-04&0.290E-10  &  0.258E-03&0.290E-09  &  0.258E-02 \\
\hline 
\end{tabular}\\[1mm]
\caption{\small Exact  value of ${\cB}_{n} f(x_1,0,\ldots,0)$, absolute error 
and relative error using $\cB_{4,0.025}^{(n)}$}\label{T1}  
\end{footnotesize}
%\end{scriptsize}
\end{table} 
In Table \ref{T2} we report on the absolute errors and approximation rates for the biharmonic potential \(\cB_n f(1,0,...,0)\) in the space dimensions \(n=5\times 10^k\), \(k=0,...,4\). The approximate values are computed by the cubature formulas \(\cB_{M,h}^{(n)}\) for \(M=1,2,3,4\). We use uniform grids of size \(h=0.1\times 2^{-k}\), \(k=1,...,5\).
 For high dimensional cases the second order formula fails whereas the eighth order formula \(\cB_{4,h}^{(n)}\) approximates with the predicted approximation rates. Table \ref{T3} shows that the cubature method is effective also for much higher space dimensions and the approximation rate is reached.
For all calculations the same quadrature rule is used for computing the one-dimensional integral, the parameters are \(\cD=5\), \(a=6\) and \(b=5\), \(\tau=0.003\) and \(300\) summands in the quadrature sum.

\begin{table}%[p]
\begin{footnotesize}
\begin{center}
$M=4$\\[1mm]

\begin{tabular}{r|cc|cc|cc|cc|cc} \hline
 $n $ & \multicolumn{2}{c|}{$5$}& \multicolumn{2}{c|}{$50$}& \multicolumn{2}{c|}{$500$}&\multicolumn{2}{c|}{$5\,000$}&\multicolumn{2}{c}{$50\,000$} \\[1pt] \hline 
$h^{-1} \hspace{-2mm}$ & error  & rate & error  & rate & error  & rate & error  &
rate & error  & rate \\[1pt] \hline       
 10   &0.15E-05 && 0.25E-04&& 0.26E-03&&0.26E-02&& 0.25E-01  &        \\ 
 20   & 0.70E-08  &    7.77& 0.11E-06  &   7.81& 0.12E-05 &    7.81 &0.12E-04  &7.81&  0.12E-03&    7.76        \\ 
 40  & 0.29E-10    &    7.94&0.46E-09& 7.95&0.47E-08  &    7.95&0.47E-07&    7.95 & 0.47E-06&    7.95       \\ 
 80   & 0.15E-12  &    7.55&0.18E-11  &7.99& 0.19E-10&    7.99&0.19E-09&    7.99& 0.19E-08&    7.99       \\ 
  160   &0.38E-13 & 2.02 & 0.10E-13  & 7.44& 0.86E-13  & 7.75& 0.84E-12&    7.80& 0.61E-11&    8.26         \\ 
\hline 
\end{tabular}\\[1mm]

$M=3$\\[1mm]

\begin{tabular}{r|cc|cc|cc|cc|cc} \hline
 10   & 0.30E-04&&0.60E-03&&0.62E-02&&0.58E-01&&0.30E+00 &        \\ 
 20   &0.53E-06   &    5.83&0.10E-04  & 5.86  &0.11E-03&    5.85&0.11E-02  & 5.74&0.11E-01  &   4.82       \\ 
 40  & 0.86E-08  &    5.96&0.17E-06  & 5.96&0.17E-05  & 5.96&0.17E-04&    5.96&0.17E-03&    5.94 \\ 
 80   & 0.13E-09 &    5.99&0.26E-08 &    5.99& 0.27E-07&    5.99& 0.27E-06&    5.99 &0.27E-05&    5.99      \\ 
  160   & 0.21E-11  & 5.97&0.41E-10  &    6.00& 0.43E-09 &    6.00 &  0.43E-08  & 6.00 & 0.43E-07&    6.00    \\ 
\hline 
\end{tabular}\\[1mm]

$M=2$\\[1mm]

\begin{tabular}{r|cc|cc|cc|cc|cc} \hline
  10   & 0.74E-03 &&0.15E-01 && 0.13E+00&&0.36E+00&&0.37E+00&        \\ 
 20   & 0.49E-04  &    3.91& 0.10E-02  & 3.89& 0.10E-01  & 3.63 &0.92E-01&    1.98&0.35E+00&    0.08       \\ 
 40  & 0.31E-05  &    3.98& 0.63E-04  & 3.98&0.67E-03 &    3.96& 0.66E-02 &    3.79&0.61E-01&    2.50  \\ 
 80   & 0.20E-06 &    3.99& 0.40E-05  &3.99& 0.42E-04  &3.99&0.42E-03 &    3.98&0.42E-02 &    3.87 \\ 
  160   & 0.12E-07 &    4.00&0.25E-06& 4.00& 0.26E-05 & 4.00 &  0.26E-04&    4.00&0.26E-03&    3.99 \\ 
  \hline 
\end{tabular}\\[1mm]

$M=1$\\[1mm]

\begin{tabular}{r|cc|cc|cc|cc|cc} \hline
10   & 0.26E-01 && 0.37E+00&&0.37E+00&&0.37E+00&&0.37E+00 &        \\ 
 20   &0.68E-02  &1.95&0.35E+00  & 0.07&0.35E+00&    0.07&0.37E+00&    0.00&  0.37E+00&    0.00        \\ 
 40  & 0.17E-02  & 1.99& 0.20E+00  &0.82& 0.20E+00&    0.82& 0.37E+00&    0.00& 0.37E+00&    0.00  \\ 
 80   & 0.43E-03 &  2.00&0.65E-01  &1.61&0.65E-01  &1.61& 0.32E+00&    0.22& 0.37E+00  &0.00 \\ 
  160   & 0.11E-03& 2.00&0.17E-01  &1.90&0.17E-01&1.90& 0.14E+00&    1.15&   0.37E+00 &    0.01  \\ 
     \hline 
\end{tabular}\\[1mm]
\caption{\small Absolute errors and approximation rates
for $\cB_n f(1,0,\ldots,0)$ using $\cB_{M,h}^{(n)}$.}\label{T2} 
\end{center}
\end{footnotesize}
\end{table} 

\begin{table}%[p]
\begin{footnotesize}
\begin{center}
$M=4$\\[1mm]
\begin{tabular}{r|cc|cc|cc} \hline
 $n $ & \multicolumn{2}{c|}{$100\,000$}& \multicolumn{2}{c|}{$1\,000\,000$}& \multicolumn{2}{c}{$10\,000\,000$}
  \\[1pt] \hline 
$h^{-1} \hspace{-2mm}$ & error  & rate & error  & rate & error  & rate  \\[1pt] \hline       
 10   &0.49E-01&&0.28E+00&&0.37E+00&      \\ 
 20  &0.23E-03&    7.71&0.23E-02 &    6.90 &0.23E-01&    4.02      \\ 
 40  &0.95E-06&    7.95&0.95E-05&    7.95& 0.95E-04&    7.91      \\ 
 80   & 0.37E-08 &    7.99& 0.37E-07&    7.99&0.37E-06  &7.99      \\ 
  160  &0.13E-10&    8.20&0.20E-09 &    7.58&0.11E-08 &8.41    \\ 
\hline 
\end{tabular}\\[1mm]

$M=3$\\[1mm]

\begin{tabular}{r|cc|cc|cc} \hline
 10  &0.36E+00&& 0.37E+00&&0.37E+00&       \\ 
  20   &0.21E-01  &4.08&0.16E+00  &1.16&0.37E+00&    0.00       \\ 
 40  & 0.35E-03&    5.92& 0.35E-02  & 5.57&0.33E-01&    3.47       \\ 
 80   &0.55E-05  &    5.99&0.55E-04&    5.98&0.55E-03  &5.92       \\ 
  160  & 0.86E-07&    6.00&0.86E-06 &    6.00 & 0.86E-05 &    6.00      \\ 
\hline 
\end{tabular}\\[1mm]
\caption{\small Absolute errors and approximation rates
for $\cB_n f(1,0,\ldots,0)$ using $\cB_{M,h}^{(n)}$.}\label{T3} 
\end{center}
\end{footnotesize}
\end{table} 

In the remainder of this section we compute the 3-dimensional biharmonic potential by means of the approximating formula \eqref{dimM3bis}. For functions of the form \eqref{sepf} we obtain that, at the points of the uniform grid \(\{h\bk\}\), the 3-dimensional integral \(\cB_3 f\) is approximated by 
\begin{align*}
\cB_3 f(h\bk)\approx &-\frac{h^4 \cD^{5/2}}{8\pi^{3/2}}\tau  \sum_{p=1}^P \beta_p
\sum_s \Phi'(\tau s) \\
& \times
 \Bigg(   \prod_{i=1}^3 \sum_{m_i}f_j^{(p)}(hm_i) 
   \ee^{-(k_i-m_i)^2/(\cD(1+\Phi(\tau s)))} Q_M\Big(\frac{k_i-m_i}{\sqrt{\cD}},\Phi(\tau s)\Big) %\right.
\\
& \quad +
\Phi(\tau s) \sum_{i=1}^3\sum_{m_i} \ee^{-(k_i-m_i)^2/(\cD(1+\Phi(\tau s)))} R_M
\Big(\frac{k_i-m_i}{\sqrt{\cD}},\Phi(\t s)\Big)f_i^{(p)}(hm_i) \\
& \qquad \times \prod_{\substack{j=1\\j\neq i}}^3 \sum_{m_j} \ee^{-(k_j-m_j)^2/(\cD(1+\Phi(\tau s)))}   
Q_M \Big(\frac{k_j-m_j}{\sqrt{\cD}},\Phi(\tau s)\Big) f_j^{(p)}(hm_j) 
\Bigg)  \, .
\end{align*}
In Table \ref{T4}  we report on the relative and absolute  errors, and the approximation rate for the 3-dimensional biharmonic potential \(\cB_3 f\) at the point \((1,1,1)\)  of the density \eqref{f(x)}, which has exact value \(\ee^{-3}\). The numerical results confirm the \(h^{2M}\) convergence of the cubature formula when \(M=1,2,3,4\).

\begin{table}%[p]
\begin{footnotesize}
\begin{center}
\begin{tabular}{r|ccc|ccc} \hline
&\multicolumn{3}{c|}{\(M=4\)}&\multicolumn{3}{c}{\(M=3\)}\\ \hline
$h^{-1} \hspace{-2mm}$ &  absolute error & relative error  &ate  &  absolute error  & relative error &
rate    \\[1pt] \hline       
 10 &  0.236E-06  &  0.474E-05  &&0.822E-05  &  0.165E-03&    \\ 
 20 &  0.965E-09  &  0.194E-07  &    7.93&0.137E-06  &  0.275E-05  &    5.91      \\ 
 40  &  0.381E-11  &  0.765E-10  &    7.99&0.217E-08  &  0.436E-07  &    5.98     \\ 
 80   &  0.150E-13  &  0.301E-12  &    7.99&0.341E-10  &  0.685E-09  &    5.99    \\ 
  160  &  0.438E-14  &  0.879E-13  &    1.77&0.538E-12  &  0.108E-10  &    5.99        \\ 
\hline 
\end{tabular}\\[1mm]

\begin{tabular}{r|ccc|ccc} \hline
&\multicolumn{3}{c|}{\(M=2\)}&\multicolumn{3}{c}{\(M=1\)}\\ \hline
$h^{-1} \hspace{-2mm}$ &  absolute error & relative error  &rate  &  absolute error  & relative error &
rate    \\[1pt] \hline       
 10   &0.217E-03  &  0.435E-02&& 0.359E-02  &  0.722E-01&      \\ 
 20   & 0.143E-04  &  0.287E-03  &    3.92& 0.925E-03  &  0.186E-01  &    1.96      \\ 
 40  & 0.907E-06  &  0.182E-04  &    3.98&0.233E-03  &  0.468E-02  &    1.99     \\ 
 80   &0.569E-07  &  0.114E-05  &    3.99&0.583E-04  &  0.117E-02  &    2.00    \\ 
  160   &0.356E-08  &  0.715E-07  &    4.00& 0.146E-04  &  0.293E-03  &    2.00        \\ 
\hline 
\end{tabular}\\[1mm]

\caption{\small Relative errors, absolute errors and approximation rates
for $\cB_3 f(1,1,1)$ using $\cB_{M,h}^{(3)}$.}\label{T4} 
\end{center}
\end{footnotesize}
\end{table} 

\eject

\end{document}